\documentclass[reqno,b5paper]{amsart}
\DeclareMathAlphabet\mathpzc{OT1}{pzc}{m}{it}
\usepackage[mathscr]{euscript}
\usepackage{color}
\usepackage{amssymb}
\usepackage{textcase}
\newcommand\rest[1]{\mathord{\restriction}_{#1}}
  \usepackage[T1]{fontenc}

\newcommand{\Ba}{\mathpzc{Ba}}
\title{On the pointwise limits of sequences of \'Swi\k{a}tkowski functions}
\author[Tomasz Natkaniec, Julia W\'odka]{Tomasz Natkaniec*, Julia W\'odka**}

\address{\llap{*\,}Institute of Mathematics, Faculty of Mathematics, Physics, and Informatics, University of Gda\'nsk, ul. Wita Stwosza 57, 80--952 Gdañsk, POLAND}
\address{\llap{**\,} \L\'od\'z University of Technology,
Institute of Mathematics,
ul.\ W\'ol\-cza\'n\-ska~21,
90--924 \L\'od\'z,
POLAND}
\email{\llap{*\,}Tomasz.Natkaniec@mat.ug.edu.pl} 
\email{\llap{**\,}JuliaWodka@gmail.com}
\subjclass[2010]{Primary 26A21; Secondary 26A15, 54C08, 54C30} 
\keywords{\'Swi\k{a}tkowski functions, cliquish functions, pointwise limits, $^\ast$topology of Hashimoto; $\mathcal{I}$-density topology; density topology}

\newtheorem{thm}{Theorem}
\newtheorem{lem}[thm]{Lemma}
\theoremstyle{definition}

\newtheorem{exa}{Example}

\newtheorem{cor}[thm]{Corollary}
\newtheorem{fact}{Fact}

\newcommand\mN{{\mathbb N}}
\newcommand{\mR}{{\mathbb R}}
\newcommand{\mQ}{{\mathbb Q}}
\newcommand{\mZ}{{\mathbb Z}}

\newcommand\co{{\mathcal C}}

\newcommand\st{:\,}

\DeclareMathOperator\Cl{cl}

\DeclareMathOperator\diam{diam}
\newcommand{\F}{\mathscr F}

\newcommand{\I}{\mathcal I}

\newcommand{\Sss}{\mathscr{S}_s}
\newcommand{\s}{\mathscr{S}}

\newcommand{\Cq}{\mathscr C_q}

\newcommand{\B}{\mathscr B}
\newcommand{\CBS}{\mathscr {CBS}}

\DeclareMathOperator{\interior}{int}
\DeclareMathOperator{\cl}{cl}
\DeclareMathOperator{\fr}{fr}
\usepackage{color}
\newcommand{\LIM}{\mathrm{LIM}}
\newcommand{\osc}{\mathrm{osc}}

\begin{document}
\begin{abstract}
The characterization of the pointwise limits of the sequences of \'Swi\k{a}tkowski functions is given. Modifications of \'Swi\k{a}tkowski property with respect to different topologies finer than the Euclidean topology are discussed.
\end{abstract}\maketitle

\section{Intrduction}
In 1977 Ma\'nk and \'Swi\k{a}tkowski defined a new property of real functions, being a kind of intermediate value property, so similar to the Darboux property: \emph{for all $a<b$ with $f(a)\ne f(b)$, there is  $x\in(a,b) \cap \co(f)$ such that $f(x)$ is between $f(a)$ and $f(b)$}~\cite{MS}.
It seems that motivations for study of such property derived from search for the weakest conditions that imply monotonicity of functions.

Ma\'nk and \'Swi\k{a}tkowski 
called mentioned condition \textit{``condition $\gamma$''} and asked if the family of Baire one Darboux functions satisfying the condition $\gamma$ are closed under the uniform limits and sums with continuous functions (both answers are positive).  Algebraic properties of the class of all functions possessing this property have been studied by many authors.
In particular, Maliszewski and the second author gave the characterization of products of \'Swi\k{a}tkowski functions -- \cite{MW} and \cite{MW1}.  W\'odka investigated the level of algebrability of the sets connected with the \'Swi\k{a}tkowski condition in the paper~\cite{JW1}.
Other results concerning \'Swi\k{a}tkowski functions can be found in \cite{FIW} and \cite{KS}. 
This note is a continuation of \cite{JW}, where uniform limits of sequences of \'Swi\k{a}tkowski functions are characterized.

\section{Preliminaries}
We shall start with some basic notations and definitions. 
We use a standard set-theoretic and topological notation. In particular, the
letter~$\mR$ denotes the real line with the Euclidean topology $\tau_e$ . Symbols $\mN$, $\mZ$, and $\mQ$ denote the sets of natural numbers, integers and rationals, respectively. 

For $A\subset \mR$ we denote by $\interior{A}$, $\cl{A}$ and $\fr{A}$ the interior, closure and boundary of $A$, respectively. 
For $a,b\in\mR$ the symbol $I(a,b)$ denotes the interval with end-points $a$ and $b$.
The set of all continuity points of a function $f:\mR\to\mR$ is denoted by~$\co(f)$.

For $x\in\mR$ and a non-empty set $A\subset\mR$ let $D(x,A)$ denote the distance between $x$ and $A\subset\mR$, i.e.,  $D(x,A):=\inf\{|x-t|\st t\in A\}$.

A set $A\subset\mR$ has \emph{the Baire property}  if there exist: an open set $O$ and a meager set $M$ such that $A=O\triangle M= (O\setminus M)\cup(M\setminus O)$. The algebra of all sets possessing the Baire property is denoted by $\mathcal{B}$. The ideal of meager sets is denoted by $\mathcal{M}$. A set $A$ is residual  if $\mR\setminus A$ is meager. We say that a set $A$ is nowhere meager in an open set $U\subset\mR$ if $A\cap W$ is non-meager for every non-empty open subset $W\subset U$.

The symbols $\osc(f,x)$ and $\osc_{\tau}(f,x)$ denote the oscillation and the oscillation with respect to a topology $\tau$ of a function $f$ at a point $x\in\mR$, respectively.

For a family $\F\subset\mR^\mR$ the symbol $\LIM(\F)$ denotes the family of all pointwise limits of the sequences $(f_n)_n$ from $\F$.

We will consider the following classes of functions from $\mR$ into $\mR$.
\begin{description}
\item[$\Ba$]
the class of all functions with the Baire property. A function $f:\mR\to\mR $ has {the Baire property} if the preimage of every open set has the Baire property.
\item[$\Cq$]
the class of all cliquish functions. We say that $f:\mR\to\mR$~is \emph{cliquish}, if for all $a<b$ and each $\varepsilon>0$ there is a nondegenerate interval $I\subset (a,b)$ such that $\diam f[I] <\varepsilon$~\cite{Thielman}. It is well-known and easy to see that $f:\mR\to\mR$ is cliquish if and only if is \emph{pointwise discontinuous}, i.e., it has a dense set of points of continuity (cf. \cite{KK, KKur}). 
\item[$\s$]
the class of all \emph{\'Swi\k{a}tkowski functions}, i.e., all functions possessing the following property: for all $a<b$ with $f(a)\ne f(b)$, there is $x\in(a,b) \cap \co(f)$ such that $f(x)\in I(f(a),f(b))$ \cite{MS}, cf. \cite{Pawlak-Pawlak, Pawlak-th}.
\item[$\Sss$]
the class of all \emph{strong \'Swi\k{a}tkowski functions}, i.e.\textcolor{red}{,} all functions with the following property:
 for all $a<b$ and each $y$ between $f(a)$ and~$f(b)$, there is $x\in (a,b) \cap \co(f)$ with $f(x)=y$~\cite{AMSss}.
\end{description}

 It is known that the following inclusions hold. (See e.g. \cite{AM}.)
\begin{gather}\label{ink}
 \Sss \subset \s\subset \Cq\subset \Ba.
\end{gather}
Moreover, easy examples show that all those inclusions are proper.

Let $I$ and $J$ be non-empty open intervals. 
We say that $f:I\to J$ is \textit{left side surjective} if $f[(\inf I, t)]=J$ for all $t\in I$. Analogously, we say that $f:I\to J$ is \textit{right side surjective} if for all  $t\in I$ we have $f[(t, \sup I)]=J$. A function $f:I\to J$ is a \textit{bi-surjective function}  if it is both left and right side surjective. If, additionally, $f$ is continuous we write $f\in\CBS(I,J)$.
A class of a continuous bi-surjective functions  plays an important role in constructions dealing with \'Swi\k{a}tkowski functions (see, e.g.,\cite{AM}, \cite{Pawlak-th}, \cite{JW}).

\section{Background informations}

It is clear that the class $\Ba$ is closed with respect to pointwise limits. The classes $\LIM(\Cq)$ and $\LIM(\Sss)$ were characterized by Grande \cite{ZG} and Maliszewski \cite{AM},  respectively. The following equalities hold.
\begin{itemize}
\item
$\LIM(\Sss)=\Cq$;
\item
$\LIM(\Cq)=\Ba$.
\end{itemize}

Those facts show that 

\[\Cq\subset \LIM(\s)\subset\Ba.\]
In this section we will show that those inclusions are proper.

\begin{exa}\label{ex1}

For $n\in\mN$ put $A_n:=\left\{\frac{k}{2^n}: k\in\mZ\right\}$ and define
\[f_n(x):=\begin{cases}
-D(x, A_n),&\text{for $x\notin A_n$},\\
1,& \text{for $x\in A_n$}.
\end{cases}\] 

One can see that for each $n\in\mN$ the function $f_n$ satisfies the \'Swi\k{a}tkowski condition. On the other hand $f_n\to f$, where
\[f=\begin{cases}
0,&\text{for $x\notin \bigcup_{n\in\mN}A_n$},\\
1,&\text{for $x\in \bigcup_{n\in\mN}A_n$}.
\end{cases}\]
Since the function $f$ is discontinuous at any point,  it is not cliquish. Hence $\LIM(\s)\ne\Cq$.
\end{exa}

\begin{exa}\label{ex2}
Consider a function $f\colon\mR\to\mR$ given by the formula
\[f(x)=\begin{cases}
0,&\text{for $x\in\mR\setminus\mQ$,}\\
1,&\text{for $x\in\mQ_1$,}\\
2,&\text{for $x\in\mQ_2$,}
\end{cases}\]
where $\mQ_1,\mQ_2$ is a partition of rationals onto two dense sets. 

It is easy to see that $f\in\Ba$.  We shall show that $f$ is not a pointwise limit of any sequence of \'Swi\k{a}tkowski functions. Indeed, suppose there exists a sequence $(f_n)_n$ of \'Swi\k{a}tkowski functions such that $f_n\to f$.  Let 
$$A_n\st =\left\{x\in\mR\setminus\mQ\st  \forall_{m\ge n}\; |f_m(x)|<\frac{1}{2}\right\}.$$
 Of course $\mR\setminus\mQ =\bigcup_n A_n$ and  $A_n\subset A_m$ for $n\leq m$. 
Since the set $\mR\setminus\mQ$ is residual, the Baire Category Theorem yields that there exists $n_0\in\mN$ for which the set $A_{n_0}$ is non-meager, and consequently, it is dense in some non-empty open interval $I$. Fix $x_1\in I\cap\mQ_1$ and $x_2\in I\cap\mQ_2$ with $x_1<x_2$.
There exist numbers $n_i$, $i=1,2$ such that for $n\geq n_i$ we have an inequality $|f_n(x_i)-i|<\frac12$. Then for $N\st =\max\{n_0, n_1, n_2\}$ we have $f_N\notin\s$. In fact, it is easy to observe that $f_N(x)\le \frac{1}{2}$ for every $x\in I\cap \co(f_N)$, thus there is no $x\in (x_1, x_2)\cap\co(f_N)$ with $f_N(x)\in (f_N(x_1), f_N(x_2))$.  Hence $\LIM(\s)\ne\Ba$.
\end{exa}

\section{The main theorem}

Fix $A\subset\mR$, an interval $J\subset\mR$ and $\varepsilon \geq 0$.  We say that a function $f:\mR\to\mR$ satisfies the condition $S(J,A,\varepsilon)$ if 
for each $a,b\in J$ with $f(a)<f(b)$ there exists $x\in A\cap I(a,b)$ such that  $f(x)\in (f(a)-\varepsilon, f(b)+\varepsilon)$. 

Note that $f\in\s$ if and only if $f$ satisfies the condition $S(\mR,\co(f), 0)$. Thus the condition $S(\mR, A, 0)$ can be treated as a generalization of the \'Swi\k{a}tkowski property related to a fixed set $A$. An analogous modification of the strong \'Swi\k{a}tkowski property has been  considered by Marciniak and Szczuka \cite{MaSz}, see also \cite{W-B-I1}.

A function $f$ satisfies the condition  $S(A, \varepsilon)$ if the union of all open intervals $J$ for which $S(J,A,\varepsilon)$ holds is dense in $\mR$.

Let $\mathfrak{S}$ denote the class of all functions $f\in\Ba$ which satisfy the condition $S(A,\varepsilon)$ for all residual sets $A\subset\mR$ and any $\varepsilon>0$.
Observe that $\Cq\subset \mathfrak{S}\subset \Ba$. We will show that $\LIM(\s)=\mathfrak{S}$.

\medskip
The next lemma is probably a part of mathematical folklore. 
\begin{lem}\label{lem1}
Let $f\in\Ba$. Then there exists a  residual $G_\delta$ set $A$ such that $f\rest A$ is continuous. Moreover, each such set $A$ can be extended to a maximal with respect to inclusion set with the same properties.
\end{lem}
\begin{proof}
It is well-known  that for every function $f\in\Ba$ there exists a residual $G_\delta$ set $A$ such that $f\rest A$ is continuous. (See e.g. \cite{KKur}.) Now, let 
$$B:=\{ x\in\mR: \osc(f\rest {A\cup\{x\}},x)=0\}.$$ 
It is easy to see that $A\subset B\in G_\delta$, $f\rest B$ is continuous, and $B$ is maximal set with those properties.
\end{proof}

\begin{lem}\label{lem2}
Let $f\in\Ba$. Then for each open interval $I$ and every $\varepsilon\geq 0$ the following conditions are equivalent:
\begin{enumerate}
\item[(i)]
$f$ satisfies the condition $S(I,A,\varepsilon)$ for every residual set $A\subset \mR$;
\item[(ii)]
$f$ satisfies the condition $S(I,A,\varepsilon)$ for every residual set $A\subset\mR$ such that $f\rest A$ is continuous;
\item[(iii)] 
there exists a residual set $A\subset\mR$ such that $f\rest A$ is continuous and $f$ satisfies the condition $S(I,A,\varepsilon)$.
\end{enumerate}
\end{lem}
\begin{proof}
Only the implication ``(iii)$\Rightarrow$(i)'' requires a proof. Let $A$ be a residual set such that $f\rest A$ is continuous and $f$ satisfies the condition $S(I,A,\varepsilon)$,  and let $B\subset\mR$ be any residual set. Fix $a,b\in I$ such that $f(a)<f(b)$. By $S(I,A,\varepsilon)$, there exists $x\in A\cap I(a,b)$ such that $f(x)\in (f(a)-\varepsilon, f(b)+\varepsilon)$. Since $A\cap B$ is dense in $\mR$ and $f\rest A$ is continuous, there exists $x_0\in A\cap B\cap I(a,b)$ with $f(x_0)\in (f(a)-\varepsilon, f(b)+\varepsilon)$.
\end{proof}
 
 \begin{cor}\label{lem1,5}
 Assume $f\in\Cq$. Then $f\in \s$ iff the condition $S(\mR,A,0)$ holds for every residual set $A\subset\mR$.
 \end{cor}

 \begin{lem}\label{lem3}
 Let $f:\mR\to\mR$ be a cliquish function. Then for any $\varepsilon>0$ and for every nowhere dense set $E\subset\mR$ there exists a maximal with respect to inclusion family of pairwise disjoint open intervals $\{ I_n: n\in\mN\}$ such that $\diam f[I_n]<\varepsilon$ for all $n\in\mN$ and end-points of any $I_n$ belong to $\mR\setminus E$. Moreover, the union of every such family is dense in $\mR$.
 \end{lem}
 \begin{proof}
 The first part is an easy consequence of the Kuratowski-Zorn Lemma. The second part follows easily from cliquishness of $f$.
 \end{proof}

\begin{thm}\label{main}
Let $f\colon\mR\to\mR$. The following conditions are equivalent:
\begin{enumerate}
\item[(i)] $f\in \LIM(\s)$;
\item[(ii)] $f\in\mathfrak{S}$.
\end{enumerate}
\end{thm}
\begin{proof}
``(i)$\Rightarrow$(ii)''
Let $f=\lim_nf_n$ for some sequence $(f_n)_n\subset\s$. Then $f$ and all $f_n$ have the Baire property, thus  there exists a residual set $A$ such that $f\rest{A}$ and all $f_n\rest{A}$ are continuous. By Lemma~\ref{lem2}, it is enough to prove that $S(A,\varepsilon)$ holds for $f$ and for any $\varepsilon>0$. Fix $\varepsilon>0$ and a non-empty open interval $I$.

For every $n\in\mN$ define 
$$A_n=\left\{x\in I\cap A\st \forall_{k\geq n} |f_k(x)-f(x)|\le \frac{\varepsilon}{3}\right\}.$$
Note that sets $A_n$ are closed in $A$ and $A\cap I=\bigcup_n A_n$. Since $A$ is non-meager, there exists $n_0$ such that $A_{n_0}$ is dense in some non-degenerate interval $J\subset I$, and consequently, $J\cap A=J\cap A_{n_0}$.
We will verify that $S(J,A,\varepsilon)$ holds. Fix $a,b\in J$ such that $f(a)<f(b)$.
Since $f_n(a)\to f(a)$ and $f_n(b)\to f(b)$, there is $N\ge n_0$ such that $f_N(a)<f_N(b)$ and
  $$|f_N(a)-f(a)|<\frac{\varepsilon}{3}  \text{ and }  |f_N(b)-f(b)|<\frac{\varepsilon}{3}.$$
By the \'Swi\k{a}tkowski property of $f_N$ and Corollary~\ref{lem1,5},
there is $x\in A_{n_0}\cap I(a,b)$ with $f_N(x)\in (f_N(a),f_N(b))$ and we have 
 $$f(a)-\varepsilon < f_N(a)-\frac{2}{3}\varepsilon< f_N(x)-\frac{2}{3}\varepsilon\le f(x)- \frac {\varepsilon}{3}<f(x),$$
 and similarly, $f(x)<f(b)+\varepsilon$. Thus $f(x)\in (f(a)-\varepsilon,f(x)+\varepsilon)$.
 
 \bigskip
``(ii)$\Rightarrow$(i)''
Let $f\in\mathfrak{S}$. 
Let $A$ be a~maximal with respect to inclusion residual $G_\delta$ set such that $f\rest{A}$ is continuous. Then $B:=\mR\setminus A=\bigcup_n B_n$, where 
$(B_n)_n$ is an increasing sequence of nowhere dense sets. 

For every $y\in\mR$ define 
$$\tilde{C}_y:=f^{-1}[y]\cap \interior(\cl(f^{-1}[y]))$$
 and set
$$C:= \bigcup_{x\in B}\tilde{C}_{f(x)}.$$
Observe that, by the maximality of $A$, for every $x\in B\cap C$ we have 
$$x\in\interior\left(\cl\left(f^{-1}[f(x)]\cap B\right)\right).$$
 For every $x\in B\cap C$ let 
$I_x$ denote the connected component of the set 
$$\interior(\cl(f^{-1}[f(x)]\cap B))$$
 containing $x$. Moreover, set $C_n=B_n\cap C$.

Observe that the set $E=\{ x\in\mR: \limsup_{t\to x, t\in A} f(t)= \pm\infty\}$ is nowhere dense.
Let us consider the following function $\tilde{f}:\mR\to\mR$.
\[\tilde{f}=\begin{cases}
f(x), &\text{for $x\in (\mR\setminus B\cap C)\cup E $,}\\
\limsup_{t\to x, t\in A} f(t), &\text{for $x\in B\cap C\setminus E$.}
\end{cases}\]
We claim that $\tilde{f}$ is cliquish. Indeed, otherwise there exist $\varepsilon>0$ and an interval $I$ such that for each $J\subset I$ we have 
$\diam \tilde{f}[J]\ge\varepsilon$. Since $f\rest{A}$ is continuous, there exists a non-empty open interval $U\subset I\setminus E$ with  $\diam f[U\cap 
A]<\frac{\varepsilon}{4}$. 
Since $\tilde{f}[J]\subset \cl(f[J])$, $\diam {f[J]}\ge\varepsilon$ for any $J\subset U$,  and consequently, there is a 
dense in $U$ set $\Delta\subset B\setminus C$  such that for every $x\in \Delta$  either $f(x)< \inf f[A\cap U]-\frac{\varepsilon}{4}$, or $f(x)> \sup f[A\cap 
U]+\frac{\varepsilon}{4}$. Let 
\begin{eqnarray*}
\Delta_- &:= &\left\{ x\in \Delta: f(x)< \inf f[A\cap U]-\frac{\varepsilon}{4}\right\};\\
\Delta_+ & := &\left\{ x\in \Delta: f(x)> \sup f[A\cap U]+\frac{\varepsilon}{4}\right\}.
\end{eqnarray*}
Fix an non-empty open interval $J\subset U$. 
Since $\Delta=\Delta_-\cup \Delta_+$, thus either $\Delta_-$ or $\Delta_+$, say $\Delta_+$, is dense in some subinterval $J_0\subset J$. 
Observe that $f$ is not constant on $J_0\cap \Delta_+$, because otherwise $J_0\cap \Delta_+ \subset C$. Fix $a,b\in J_0\cap \Delta_+$ with $f(a)<f(b)$ and observe that $f(x)\le f(a)-\frac{\varepsilon}{4}$ for any $x\in A\cap J_0$, hence $f[A\cap J_0]\cap (f(a), f(b))=\emptyset$. Thus for any subinterval $J\subset U$ the function $f$ does not satisfy the condition $S(J,A,\frac{\varepsilon}{8})$, 
contrary with $f\in\mathfrak{S}$.

\bigskip
Since $\tilde{f}\in\Cq$, there exists a sequence $(\tilde{f}_n)_n$ of strong \'Swi\k{a}tkowski functions such that $\tilde{f}_n\to\tilde{f}$ \cite[Corollary 6.]{AM}.
We will modify functions $\tilde{f}_n$ using a~standard trick with bi-surjective functions (see, e.g.~\cite{JW}). So,  for every $n\in\mN$ let $\{I_m^n: m\in\mN\}$ be a maximal family (with respect to inclusion) of pairwise disjoint  open intervals contained in the set $\mR\setminus B_n$ with end-points in $A$ and $\diam \tilde{f_n}[I_m^n] <\frac{1}{4n}$.
Since $\tilde{f}_n$ is cliquish, Lemma~\ref{lem3} yields   $\Cl(\bigcup_{m\in\mN} I_m^n)=\mR$. Next, for every $m\in\mN$ with $I_m^n\ne\emptyset$ choose $x_m^n\in I_m^n\cap A$. Finally, choose $\tilde{f}_m^n\in\CBS( I_m^n, (\tilde{f_n}(x^n_m)-\frac{1}{2n}, \tilde{f}_n(x^n_m)+\frac{1}{2n}))$ and define a function $g_n$. 

\[g_n(x)=\begin{cases}
\tilde{f}_m^n(x), &\text{for $x\in I_m^n$, $m\in\mN$,}\\
f(x), &\text{for $x\in C_n$,}\\
\tilde{f_n}(x), &\text{in other cases}.
\end{cases}\]

We claim that if $f(x)\ne f(y)\ne f(z)\ne f(x)$ for some $x,y,z\in B\cap C$, then $I_x\cap I_y\cap I_z=\emptyset$. Indeed, suppose that $I_x\cap I_y\cap I_z\ne\emptyset$. First observe that there exists $a\in A\cap I_x\cap I_y\cap I_z$ such that $f(a)\notin\{f(x), f(y), f(z)\}$.
In fact, suppose that $f[A\cap I_x\cap I_y\cap I_z]\subset\{f(x),f(y),f(z)\}$. Then there is an interval $J\subset I_x\cap I_y\cap I_z$ with $f$ being constant on $J\cap A$, say $f(a)=f(x)$ for $a\in J\cap A$. Choose  $s\in J\cap f^{-1}(f(x))\cap B$. Then $f\rest {(A\cup\{ s\})}$ is continuous, contrary to the maximality of $A$. Hence  $f(a)\notin\{f(x), f(y), f(z)\}$.
Take 
$$\varepsilon :=\frac{ \min\left\{|f(a)-f(x)|, |f(a)-f(y)|, |f(a)-f(z)|\right\}}{2}.$$
 Let $U\subset I_x\cap I_y \cap I_z$ be a neighborhood of $a$ such that $\diam f[A\cap U]<\varepsilon$ and let $J$ be any subinterval of $U$. Choose $\bar{x},\bar{y},\bar{z}\in U$ such with$f(\bar{t})=f(t)$ for $t\in\{x,y,z\}$.  Changing, possibly, the names we can assume that $f(a)<f(\bar{x})<f(\bar{y})$. Then  
$$\left(f(\bar{x})-\varepsilon, f(\bar{y})+\varepsilon\right)\cap f\left[I\left(\bar{x}, \bar{y}\right)\cap A\right] =\emptyset.$$
Thus $f$ satisfies the condition $S(J,A,\varepsilon)$ for no subinterval $J\subset U$, contrary to $f\in\mathfrak{S}$.

\bigskip
In the next step we claim that the set 
\[D:= \left\{ \inf I_x: x\in B\cap C \right\}\cap\mR\]
is nowhere dense. Indeed, take an arbitrary open set $U$ and $x\in U\cap D$. If there is no $y\in D\cap U\cap I_x$ then $I_x\cap U$ is a non-empty open set disjoint with $D$. Otherwise, by the previous claim, $D$ is disjoint with $I_x\cap I_y$.

Let $\I$ denote the family of connected components of the set $\mR\setminus \cl(D)$. For $I\in\I$ and $n\in\mN$ fix a function $f_n^I\in\CBS((\inf I, \inf I+\frac{|I|}{n}), \mR)$. 
Finally, for every $n\in\mN$ define a function $f_n:\mR\to\mR$ by

\[f_n(x)=\begin{cases}
f_n^I(x), &\text{for $x\in (\inf I, \inf I+\frac{|I|}{n})$, $I\in\I$,}\\
g_n, &\text{in the oposite case.}
\end{cases}\]

Fix $n\in\mN$. We shall show that $f_n\in\s$. Take $a<b$ and assume that $f_n(a)<f_n(b)$ (the second case is analogous). We have to consider a few cases.

\bigskip
\textsc{Case I.}  
$(a,b)\cap\cl(D)\ne\emptyset$. Let $I$ be a component of $\mR\setminus\cl(D)$ with $\inf I\in (a,b)$. Then for every $y\in\mR$ there is $t\in (a,b)\cap I$ such that $f_n(t)=f_n^I(t)=y$, thus the \'Swi\k{a}tkowski condition is satisfied.

\bigskip
\textsc{Case II.}
$(a,b)\cap\cl(D)=\emptyset$ and $(a,b)\cap (\inf I, \inf I+\frac{|I|}{n})\ne\emptyset$ for some component $I$ of $\mR\setminus\cl(D)$. Then we have two subcases.

\medskip
\textsc{II.1.}
$(a,b)\subset (\inf I, \inf I+\frac{|I|}{n})$. Then $f_n$ agrees with $f_n^I$ on $(a,b)$, thus it is continuous on $(a,b)$, so the \'Swi\k{a}tkowski condition holds.

\medskip
\textsc{II.2.} 
$\inf I+\frac{|I|}{n}\in (a,b)$. Then for every $y\in\mR$ there is $t\in (a,b)\cap I$ such that $f_n(t)=f_n^I(t)=y$.

\bigskip
\textsc{Case III.}
$(a,b)\subset I\setminus (\inf I, \inf I+\frac{|I|}{n})$ for for some component $I$ of $\mR\setminus\cl(D)$. Then $f_n(x)=g_n(x)$ for $x\in (a,b)$.
Two subcases may occur.

\medskip
\textsc{III.1.}
$\{ a,b\}\not\subset C_n$. We may assume that $a\not\in C_n$ (the other case is analogous). 
If $a\in I_n^m$ for some $m\in\mN$, then the \'Swi\k{a}tkowski condition easily holds. If not, $f_n(a)=g_n(a)=\tilde{f}_n(a)$, and, 
by the strong \'Swi\k{a}tkowski property of $\tilde{f}_n$ there exists $t\in (a,b)$ such that $\tilde{f}_n(t)\in (f_n(a)-\frac{1}{8n}, f_n(b)+\frac{1}{8n})$. By the construction of the family $(f_n^m)_m$ there exists $m\in\mN$ and $s\in I_n^m$ such that $I_n^m\subset (a,b)$ and  $\tilde{f}_n^m(s)\in (f_n(a), f_n(b))$. Since $ \tilde{f}_n^m(s)=f_n(s)$ and $I_n^m\subset\co(f_n)$ the \'Swi\k{a}tkowski condition holds.

 \medskip
 \textsc{III.2.} $\{ a,b\}\subset C_n$. Then $\inf I_b<a<b$ and  $f_n(t)=g_n(t)=f(t)$ for $t\in\{ a,b\}$, so $f(a)<f(b)$. For $\varepsilon=\frac{1}{8n}$ there is an open interval $J\subset (a,b)\cap I_a\cap I_b$ such that $f$ satisfies the condition $S(J,A,\varepsilon)$. We may assume that $J\subset I^n_m$ for some $m$. Fix $x_a\in J\cap f^{-1}[f(a)]$ and $x_b\in J\cap f^{-1}[f(b)]$. By $S(J,A,\varepsilon)$, there exists ${x}\in J\cap A$ with $f({x})\in (f(a)-\varepsilon, f(b)+\varepsilon)$. But then there exist $m\in\mN$ and $t\in I_n^m$ such that $I_n^m\subset J$ and $f_n(t)=g_n(t)=\tilde{f}^n_m(t)$, thus $t\in\co(f_n)$ and $f_n(t)\in (f_n(a),f_n(b))$, so the \'Swi\k{a}tkowski condition is satisfied.  

 \bigskip
Finally observe that $f_n\to f$. In fact, fix $x\in\mR$. Observe that if $x\not\in B\cap C$ then $|g_n(x) - \tilde{f}_n(x)|<\frac{1}{n}$, thus $$\lim_{n\to\infty}g_n(x)=\lim_{n\to\infty}\tilde{f}_n(x)=\tilde{f}(x)=f(x).$$ Moreover, there exists $n_0$ such that $f_n(x)=g_n(x)$ for $n>n_0$, hence $\lim_nf_n(x)=f(x)$. If $x\in B\cap C$ then there is $n_0$ such that $x\in C_n$ for $n>n_0$ and then $f_n(x)=g_n(x)=f(x)$.
 \end{proof}

\begin{cor}
If $B\subset\mR$ is a set with the Baire property then the characteristic function of $B$ is a pointwise limit of a sequence of \'Swi\k{a}tkowski functions.
\end{cor}
\begin{proof}
Let $\chi_B$ be the characteristic function of $B$.
Since $B\in\B$, 
there are an open set $U$ and a meager set $M$ with $B=U\triangle M$. 
Let $A=\mR\setminus (M\cup \fr(U))$. Clearly $A$ is residual and $\chi_B\rest{A}$ 
is continuous. Fix $a,b\in\mR$ with $\chi_B(a)<\chi_B(b)$. Then $\chi_B(a)=0$ and $\chi_B(b)=1$, so for any $\varepsilon>0$ we have $\chi_B[\mR]=\{0,1\}\subset (\chi_B(a)-\varepsilon,\chi_B(b)+\varepsilon)$, hence $\chi_B$ satisfies the condition $S(A,\varepsilon)$.
\end{proof}

\section{Baire system generated by the family $\s$}
For a family $\mathcal{F}$ of real-valued functions defined on $X$ there is a smallest
family $\mathcal{B}(\mathcal{F})$ of all real-valued functions defined on $\mR$
which contains $\mathcal{F}$ and which is closed under the process of taking
limits of sequences. This family is called the \emph{Baire
system} generated by $\mathcal{F}$. For a given $\mathcal{F}$  let us define
\begin{itemize}
\item
$\mathcal{B}_0(\mathcal{F})=\mathcal{F}$;
\item
$\mathcal{B}_\alpha(\mathcal{F})=\LIM\left(\bigcup_{\beta<\alpha}\mathcal{B}_\beta(\mathcal{F})\right)$
for $\alpha>0$.
\end{itemize}
Then $\mathcal{B}(\mathcal{F})=\bigcup_{\alpha<\omega_1}\mathcal{B}_\alpha(\mathcal{F})$.
This system was describe in 1899 by Baire in
the case when $\mathcal{F}$ is the
family of all continuous functions. The minimal ordinal $\alpha\le \omega_1$ with $\mathcal{B}_{\alpha}(\mathcal{F})=\bigcup_{\beta<\alpha}\mathcal{B}_\beta(\mathcal{F})$ is called \emph{Baire order} of the family $\mathcal{F}$.

\begin{thm}
We have
\begin{enumerate}
\item[(i)]
$\mathcal{B}_1(\s)=\mathfrak{S}$;
\item[(ii)]
$\mathcal{B}_\alpha(\s)=\Ba$ for $\alpha>1$.
\end{enumerate}
Hence the Baire order of the family of all \'Swi\k{a}tkowski functions is equal to 2.
\end{thm}
\begin{proof}
The equality (i) follows from Theorem~\ref{main}. Since $\mathfrak{S}\subset\Ba$ and the class of all functions possessing the Baire property is closed with respect to pointwise limits, $\mathcal{B}_2(\s)=\LIM(\mathfrak{S})\subset\Ba$. Since $\Ba\subset \LIM(\Cq)$ \cite{ZG} and $\Cq\subset\mathfrak{S}$, $\Ba\subset\mathcal{B}_2(\s)$. Again, since $\LIM(\Ba)=\Ba$, we have $\mathcal{B}_\alpha(\s)=\Ba$ for $\alpha\ge 1$.
\end{proof}

\section{A generalization: $\tau$-\'Swi\k{a}tkowski functions}
In this section we will consider a slight generalization of the \'Swi\k{a}tkowski property. Let $\tau$ be a fixed topology on $\mR$. 
For a function $f:\mR\to\mR$ let $\co_\tau(f)$ denote the set of all points $x$ at which $f$ is continuous as a function from the space $(\mR,\tau)$ into $\mR$ with the Euclidean topology.

We say that a function $f:\mR\to\mR$ has the \'Swi\k{a}tkowski property with respect to $\tau$ (shortly, $f$ is a $\tau$-\'Swi\k{a}tkowski function) if 
$f\in S(\co_\tau(f),0)$. The class of all $\tau$-\'Swi\k{a}tkowski functions will be denoted by $\s_\tau$.

Note that an analogous modification of the strong \'Swi\k{a}tkowski property has been considered in \cite{ZG1} and \cite{W-B-I}.

\begin{thm}\label{thm-tau}
Let $\tau$ be topology on $\mR$ satisfying the following conditions:
\begin{enumerate}
\item[(i)]
$\tau$ is finer than the Euclidean topology: $\tau_e\subset\tau$;
\item[(ii)]
$\tau\setminus\{\emptyset\}\subset \mathcal{B}\setminus \mathcal{M}$.
\end{enumerate}
Then $\LIM(\s_\tau)=\mathfrak{S}$.
\end{thm}
\begin{proof}
``$\supset$'' Since $\tau_e\subset\tau$, so $\s\subset\s_\tau$ and consequently, $\mathfrak{S}\subset\LIM(\s_\tau)$.

\noindent
``$\subset$'' First let us see that $\s_\tau\subset\Ba$. We will use the following well-known fact.
\begin{fact}
If $f\not\in\Ba$ then there are reals $\alpha<\beta$ such that the sets $A=f^{-1}[(-\infty,\alpha)]$, $B=f^{-1}[(\beta,\infty)]$ are both nowhere meager in some non-empty open set $U\subset \mR$.
\end{fact}
Thus if $f\not\in\Ba$ then the set $\co_\tau(f)$ is not dense in $\mR$ and therefore $f$ is not a~$\tau$-\'Swi\k{a}tkowski function.

Now assume that $f\in\LIM(\s_\tau)$, i.e. there exists a sequence $(f_n)_n$ of $\tau$-\'Swi\k{a}tkowski functions tending to $f$. Then $f\in\Ba$, so there is a residual set $A$ such that $f\rest A$ is continuous.  By Lemma~\ref{lem2}, it is enough to prove that $S(A,\varepsilon)$ holds for any $\varepsilon>0$. Fix $\varepsilon>0$ and a non-empty open interval $I$.
For every $n\in\mN$ define 
$$A_n=\left\{x\in I\cap A\st \forall_{k\geq n} |f_k(x)-f(x)|<\frac{\varepsilon}{3}\right\}.$$
Observe that each $A_n$ has the Baire property, $A\cap I=\bigcup_n A_n$ and $A_n\subset A_{n+1}$. Since $A$ is non-meager, there exists $n_0$ such that $A_{n_0}$ is residual in some non-degenerate interval $J\subset I$. We will verify that $S(J,A,\varepsilon)$ holds. Fix $a,b\in J$ such that $f(a)<f(b)$.
Since $f_n(a)\to f(a)$ and $f_n(b)\to f(b)$, there is $N\ge n_0$ such that $f_N(a)<f_N(b)$ and
  $$|f_N(a)-f(a)|<\frac{\varepsilon}{3}  \text{ and }  |f_N(b)-f(b)|<\frac{\varepsilon}{3}.$$
 By the $\tau$-\'Swi\k{a}tkowski property of $f_N$, there is $x\in\co_\tau(f_N)\cap I(a,b)$ with $f_N(x)\in (f_N(a),f_N(b))$. Let $V\subset I(a,b)$ be a $\tau$-neighborhood of $x$ such that $f_N[V]\subset (f_N(a), f_N(b))$. Since $V\cap A_{n_0}\ne\emptyset$, there exists 
$x_0\in A_{n_0}\cap I(a,b)$ with $f_N(x_0)\in (f_N(a),f_N(b))$ and we have 
 $$f(a)-\varepsilon<f_N(a)-\frac{2}{3}\varepsilon< f_N(x_0)-\frac{2}{3}\varepsilon<f(x_0)- \frac {\varepsilon}{3}<f(x_0),$$
 and similarly, $f(x_0)<f(b)+\varepsilon$. Thus $f(x_0)\in (f(a)-\varepsilon,f(b)+\varepsilon)$.
 \end{proof}

In particular, the following two topologies satisfy assumptions of Theorem~\ref{thm-tau}: 
\begin{description}
\item[$\tau_\ast$]
 the $\ast$-topology of Hashimoto with respect to the ideal $\mathcal{M}$~\cite{HH}. Recall that $$\tau_\ast=
\{ U\setminus M: U\in\tau_e; \; M\in\mathcal{M}\}.$$
\item[$\tau_\mathcal{I}$]
the $\mathcal{I}$-density topology, the category counterpart of the density topology~\cite{PWW}.
\end{description}
Recall that $\tau_e\subset\tau_\ast\subset\tau_\mathcal{I}$ and both inclusions are here proper. 

\begin{exa}\label{exa3}
Let $\mQ_1,\mQ_2$ be a partition of rationals onto two dense sets. Define $Q_i:=\mQ_i\cup\{(\frac{5-2i}{2}+2k)\cdot\pi: k\in\mZ\}$. Then the function
\[f(x)=\begin{cases}
\sin(x),&\text{for $x\in\mR\setminus (Q_1\cup Q_2)$,}\\
(-1)^i2,&\text{for $x\in\ Q_i$, $i=1,2$}
\end{cases}\]
is a $\tau_\ast$-\'Swi\k{a}tkowski function which is not \'Swi\k{a}tkowski.
\end{exa}

\begin{cor}
Although  the families $\s$, $\s_{\tau_\ast}$ are different, their pointwise closures coincide.
\end{cor}

\begin{lem}\label{lem-I-d}
Assume $\tau$ is a topology on $\mR$ which satisfies assumptions of Theorem~\ref{thm-tau}. For any function $f:\mR\to\mR$, if the set $\co_\tau(f)$ is dense, then it is residual. 
\end{lem}
\begin{proof}
For $n\in\mN$ define 
$$G_n=\{ x\in\mR: \osc_\tau(f,x)<\frac{1}{n}\}.$$
Clearly, $G_n\in\tau$, hence 
$\co_\tau(f)=\bigcap_{n\in\mN}G_n$ is a $G_\delta$ set in the topology $\tau$, so it
has the Baire property, thus it is enough to prove  that $\co_\tau(f)\cap U\not\in\mathcal{M}$ for every non-empty open set $U\in\tau_e$.
Fix $U\in\tau_e$ and $n\in\mN$. Since $\co_\tau$ is dense, $G_n\cap U$ is $\tau$-open and non-empty, hence $G_n\cap U\not\in\mathcal{M}$. Therefore each $G_n$ is residual, thus $\co_\tau(f)$ is residual too.
\end{proof}

\begin{thm}\label{thm-I-density}
Assume $\tau$ is a topology on $\mR$ which satisfies assumptions of Theorem~\ref{thm-tau}. Then if $\tau_\ast\subset\tau$ then $\s_{\tau_\ast}=\s_{\tau}$.
\end{thm}
\begin{proof}
``$\subset$'' Since $\tau_\ast\subset\tau$, we have $\s_{\tau_\ast}\subset\s_\tau$. 

``$\supset$'' Assume $f:\mR\to\mR$ has the $\tau$-\'Swi\k{a}tkowski property. In the first part of the proof of Theorem~\ref{thm-tau} it is shown that 
$f\in\Ba$, hence there is a residual set $C\subset\mR$ for which $f\rest C$ is continuous. By Lemma~\ref{lem-I-d}, the set $D:=C\cap\co_\tau(f)$ is 
residual. Then $f\rest D$ is continuous and therefore $f$ is $\tau_\ast$-continuous at each point $x\in D$. To prove that $f$ has the $\tau_\ast$-\'Swi\k{a}tkowski property fix $a,b\in \mR$ with $f(a)<f(b)$. Since $f\in\s_\tau$, there is $x\in \co_\tau(f)\cap I(a,b)$ with $f(x)\in (f(a),f(b))$. Let 
$U\in\tau$ be a $\tau$-neighborhood of $x$ such that $U\subset I(a,b)$ and $f[U]\subset (f(a), f(b))$. Then $U\not\in\mathcal{M}$, hence $U\cap D\ne\emptyset$
, so there is $x_0\in D\subset\co_{\tau_\ast}(f)$ such that $x_0\in I(a,b)$ and $f(x_0)\in (f(a),f(b))$.
\end{proof}

\begin{cor}
We have $\s_{\tau_\mathcal{I}}=\s_{\tau_\ast}$.
\end{cor}

Finally we will discuss the $\tau$-\'Swi\k{a}tkowski property related to some topologies $\tau$ connected with the Lebesgue measure on the real line.
Interestingly,  an analog of Theorem~\ref{thm-I-density} for the measure does not occur. (We obtain a new example of an incomplete duality between the measure and category.) Let us consider the following topologies:
\begin{description}
\item[$d_\ast$]
 the $\ast$-topology of Hashimoto with respect to the ideal $\mathcal{N}$ of Lebesgue nullsets. Recall that $$d_\ast=
\{ U\setminus N: U\in\tau_e; \; N\in\mathcal{N}\}.$$
\item[$d$]
the density topology, see e.g. \cite{JO}.
\end{description}
Recall that $\tau_e\subset d_\ast\subset d$ and both inclusions are here proper. Thus 
$$\s\subset \s_{d_\ast} \subset \s_d.$$
The function $f$ from Example~\ref{exa3} shows that the inclusion $\s\subset \s_{d_\ast}$ is proper.
The next example shows that the inclusion $\s_{d_\ast} \subset \s_d$ is proper, too.

\begin{exa}
Let $E_1,E_2\subset\mR$ be disjoint $F_\sigma$ set such that for each non-degenerate interval $J\subset\mR$ the sets $J\cap E_i$, $i=1,2$, and $J\setminus (E_1\cup E_2)$ have positive measure (cf. \cite[Section~8]{JO}). Moreover, assume that $\{(\frac{5-2i}{2}+2k)\cdot\pi: k\in\mZ \}\subset E_i$ for $i=1,2$. Then the function
\[f(x)=\begin{cases}
\sin(x),&\text{for $x\in\mR\setminus (E_1\cup E_2)$,}\\
(-1)^i2,&\text{for $x\in\ E_i$, $i=1,2$}
\end{cases}\]
is a $d$-\'Swi\k{a}tkowski function which is not $d_\ast$-\'Swi\k{a}tkowski.
\end{exa}



\begin{thebibliography}{10}
\bibitem{FIW} M. Filipczak, G. Ivanova, J. W\'odka, \textit{Comparison of some families of real functions in porosity terms}, Math. Slovaca, to appear. 
\bibitem{ZG} Z. Grande,\textit{ Sur la quasi-continuit\'e et la quasi-continuit\'e approximative}, Fund. Math., \textbf{129}, (1988), 167--172, MR0962538,  Zbl 0657.26003.
\bibitem{ZG1} Z. Grande,   \textit{On a subclass of the family of Darboux functions}, Colloq. Math. \textbf{117}, (2009), 95--104, MR2539550, Zbl 1177.26005.
\bibitem{HH}
H. Hashimoto, \textit{On the $^\ast$topology and its application}, Fund. Math., \textbf{91}, (1976), 5--10, MR0413058, Zbl 0357.54002.
\bibitem{W-B-I}
G. Ivanova, E. Wagner-Bojakowska, \textit{On some modification of Darboux property}, Math. Slovaca \textbf{66}, (2016), no. 1, 79--88, MR3510852, Zbl 06589831.
\bibitem{W-B-I1}
G. Ivanova, E. Wagner-Bojakowska, \textit{On some subclasses of the family of Darboux Baire 1 functions}, Opuscula Math. \textbf{34} No. 4, (2014), 777--788, MR3283017, Zbl 1339.26013.

\bibitem{KS} M. Kowalewski, P. Szczuka, \textit{Separating sets by \'Swi\k{a}tkowski functions}, Quaestiones Mathematicae, \textbf{39}(4) 2016, 471--477, MR3521167, Zbl 1119.26008.

\bibitem{KKur}
K. Kuratowski, \textit{Topologie}, Vol. I, PWN, Warszawa 1958, MR1296876,  Zbl 0158.40901.
\bibitem{KK} R. Lester,  \textit{Pointwise discontinuous function}, dissertation,  University of Missouri, 1912.

\bibitem{AM} A. Maliszewski, \textit{Darboux Property and Quasi-continuity. A Uniform Approach}, dissertation, WSP, S\l{}upsk 1996.
\bibitem{AMSss} A. Maliszewski, \textit{On the limits of strong \'Swi\k{a}tkowski functions}, Zeszyty Nauk. Politech. \L\'odz. Mat. \textbf{27} (1995), no. 719, 87--93, MR1357159, Zbl 0885.26002.
\bibitem{MW} A. Maliszewski, J. W\'odka,
 \textit{Products of \'Swi\k{a}tkowski functions}, Math. Slovaca, 66(3)(2016), 601--604,MR3543724, Zbl 06639573.
\bibitem{MW1} A. Maliszewski, J. W\'odka,
 \textit{Products of \'Swi\k{a}tkowski and quasi-continuous functions}, J. Appl. Anal., 20 (2)	(2014), 129--132, MR3284719, Zbl 1305.26012.
\bibitem{MS}
T. Ma\'nk, T. \'Swi\k{a}tkowski, \textit{On some class of functions with Darboux's characteristic}, Zeszyty Nauk. Politech. \L\'odz. Mat. \textbf{11} (1977), no. 301, 5--10, MR0633328,  Zbl 0416.26005.
\bibitem{MaSz}
M. Marciniak, P. Szczuka, \textit{$A$-Darboux functions}, Lith. Math. J. \textbf{56}, no. 1, (2016), 107-113, MR3472109, Zbl 1342.26011.
\bibitem{JO}
J. Oxtoby, \textit{Measure and Category}, Springer 1980, MR0584443, Zbl 0435.28011.
\bibitem{Pawlak-th}
R.J. Pawlak,  \textit{Przekszta\l{}cenia Darboux}, dissertation, \L\'od\'z  University, 1985, (in Polish).
\bibitem{Pawlak-Pawlak}
H. Pawlak, R.J. Pawlak, \textit{On some conditions equivalent to the condition of \'Swi\k{a}tkowski for Darboux functions of one and two variables}, 
Zeszyty Nauk. Politech. \L\'odz. Mat. \textbf{16} (1983), no. 413, 33--40, MR0744182, Zbl 0597.26003.
\bibitem{PWW}
W. Poreda, E. Wagner-Bojakowska, W. Wilczy\'nski, \textit{A category analogue of the density topology}, Fund. Math. \textbf{125}, (1985), 167--173, MR0813753, Zbl 0613.26002.
\bibitem{Thielman}
H.P. Thielman, \emph{Types of functions}, Amer. Math. Monthly \textbf{60} (1953), 156--161, MR0052495  Zbl 0051.13801.
\bibitem{JW} 
J. W\'odka, \textit{On the uniform limits of sequences of \'Swi\k{a}tkowski functions}, submitted.
\bibitem{JW1} J. W\'odka, \emph{Subsets of some families of real functions and their algebrability}, Linear Algebra Appl., 459	(2014), 454-- 464, MR3247237,  Zbl 1309.15005.


\end{thebibliography}
\end{document}